\documentclass{amsart}
\usepackage{amsmath}
\usepackage{amsfonts}

\newtheorem{theorem}{Theorem}
\theoremstyle{plain}

\newtheorem{corollary}{Corollary}

\newtheorem{definition}{Definition}

\newtheorem{lemma}{Lemma}

\newtheorem{remark}{Remark}

\numberwithin{equation}{section}
\input{tcilatex}

\begin{document}
\title[On generalization of different type inequalities]{Generalization of
different type integral inequalities for $(\alpha ,m)-$convex functions via
fractional integrals}
\author{\.{I}mdat \.{I}\c{s}can}
\address{Department of Mathematics, Faculty of Sciences and Arts, Giresun
University, Giresun, Turkey}
\email{imdat.iscan@giresun.edu.tr; imdati@yahoo.com}

\begin{abstract}
In this paper, a general integral identity for a twice differentiable
functions is derived. By using of this identity, the author establishes some
new Hermite-Hadamard type and Simpson type inequalities for differentiable $%
(\alpha ,m)-$convex functions via Riemann Liouville fractional integral.
\end{abstract}

\maketitle

\section{Introduction}

Let $f:I\subseteq \mathbb{R\rightarrow R}$ be a function defined on the
interval $I$ of real numbers. Then $f$ is called convex if%
\begin{equation*}
f\left( tx+(1-t)y\right) \leq tf(x)+(1-t)f(y)
\end{equation*}%
for all $x,y\in I$ and $t\in \lbrack 0,1].$ There are many results
associated with convex functions in the area of inequalities, but one of
those is the classical Hermite Hadamard inequality:%
\begin{equation}
f\left( \frac{a+b}{2}\right) \leq \frac{1}{b-a}\dint\limits_{a}^{b}f(x)dx%
\leq \frac{f(a)+f(b)}{2}\text{.}  \label{1-1}
\end{equation}%
$a,b\in I$ with $a<b$. The another important and well known inequality is
Simpson's inequality. This inequality is stated as:

Let $f:\left[ a,b\right] \mathbb{\rightarrow R}$ be a four times
continuously differentiable mapping on $\left( a,b\right) $ and $\left\Vert
f^{(4)}\right\Vert _{\infty }=\underset{x\in \left( a,b\right) }{\sup }%
\left\vert f^{(4)}(x)\right\vert <\infty .$ Then the following inequality
holds:%
\begin{equation*}
\left\vert \frac{1}{3}\left[ \frac{f(a)+f(b)}{2}+2f\left( \frac{a+b}{2}%
\right) \right] -\frac{1}{b-a}\dint\limits_{a}^{b}f(x)dx\right\vert \leq 
\frac{1}{2880}\left\Vert f^{(4)}\right\Vert _{\infty }\left( b-a\right) ^{4}.
\end{equation*}

In \cite{M93}, V.G. Mihesan presented the class of $(\alpha ,m)$-convex
functions as below:

\begin{definition}
The function $f:\left[ 0,b\right] \mathbb{\rightarrow R}$, $b>0$, is said to
be $(\alpha ,m)$-convex where $(\alpha ,m)\in \left[ 0,1\right] ^{2}$, if we
have%
\begin{equation*}
f\left( tx+m(1-t)y\right) \leq t^{\alpha }f(x)+m(1-t^{\alpha })f(y)
\end{equation*}
\end{definition}

for all $x,y\in \left[ 0,b\right] $ and $t\in \left[ 0,1\right] $.

Note that for $(\alpha ,m)\in \left\{ (0,0),(\alpha
,0),(1,0),(1,m),(1,1),(\alpha ,1)\right\} $ one obtains the following
classes of functions respectively: increasing, $\alpha $-starshaped,
starshaped, $m$-convex, convex, $\alpha $-convex functions.

Denote by $K_{m}^{\alpha }(b)$ the set of all $(\alpha ,m)$-convex functions
on $\left[ 0,b\right] $ for which $f(0)\leq 0$. For recent results and
generalizations concerning $(\alpha ,m)$-convex functions see \cite%
{I13,I13b,OAK11,P12b,WLFZ12}.

We give some necessary definitions and mathematical preliminaries of
fractional calculus theory which are used throughout this paper.

\begin{definition}
Let $f\in L\left[ a,b\right] $. The Riemann-Liouville integrals $%
J_{a^{+}}^{\kappa }f$ and $J_{b^{-}}^{\kappa }f$ of oder $\kappa >0$ with $%
a\geq 0$ are defined by

\begin{equation*}
J_{a^{+}}^{\kappa }f(x)=\frac{1}{\Gamma (\kappa )}\dint\limits_{a}^{x}\left(
x-t\right) ^{\kappa -1}f(t)dt,\ x>a
\end{equation*}

and

\begin{equation*}
J_{b^{-}}^{\kappa }f(x)=\frac{1}{\Gamma (\kappa )}\dint\limits_{x}^{b}\left(
t-x\right) ^{\kappa -1}f(t)dt,\ x<b
\end{equation*}%
respectively, where $\Gamma (\kappa )$ is the Gamma function defined by $%
\Gamma (\kappa )=$ $\dint\limits_{0}^{\infty }e^{-t}t^{\kappa -1}dt$ and $%
J_{a^{+}}^{0}f(x)=J_{b^{-}}^{0}f(x)=f(x).$
\end{definition}

In the case of $\kappa =1$, the fractional integral reduces to the classical
integral. Recently, many authors have studied a number of inequalities by
used the Riemann Liouville fractional integrals, see \cite%
{I13c,SO12,SSYB11,S12,SY13}\ and the references cited therein.

In \cite{SA11}, Sarikaya et.al established some new inequalities of the
Simpson and the Hermite--Hadamard type for functions whose absolute values
of derivatives are convex:

\begin{theorem}
Let $I\subset 
\mathbb{R}
$ be an open interval, $a,b\in I$ with $a<b$\ and $f:I\rightarrow 
\mathbb{R}
$ be a twice differentiable mapping on $I^{\circ }$ such that $f^{\prime
\prime }$ is integrable and $0\leq \lambda \leq 1.$ If $|f^{\prime \prime
}|^{q}$ is a convex mapping on $[a,b]$, $q\geq 1$, then the following
inequalities hold:%
\begin{equation}
\left\vert \left( 1-\lambda \right) f\left( \frac{a+b}{2}\right) +\lambda 
\frac{f(a)+f(b)}{2}-\frac{1}{\left( b-a\right) }\dint\limits_{a}^{b}f(x)dx%
\right\vert  \label{1-1.1}
\end{equation}%
\begin{equation*}
\leq \left\{ 
\begin{array}{c}
\frac{\left( b-a\right) ^{2}}{2}\left( \frac{\lambda ^{3}}{3}+\frac{%
1-3\lambda }{24}\right) ^{1-\frac{1}{q}} \\ 
\times \left\{ \left( \left[ \frac{\lambda ^{4}}{6}+\frac{3-8\lambda }{%
3.2^{6}}\right] \left\vert f^{\prime \prime }(a)\right\vert ^{q}+\left[ 
\frac{\left( 2-\lambda \right) \lambda ^{3}}{6}+\frac{5-16\lambda }{3.2^{6}}%
\right] \left\vert f^{\prime \prime }(b)\right\vert ^{q}\right) ^{\frac{1}{q}%
}\right. \\ 
\left. +\left( \left[ \frac{\lambda ^{4}}{6}+\frac{3-8\lambda }{3.2^{6}}%
\right] \left\vert f^{\prime \prime }(b)\right\vert ^{q}+\left[ \frac{\left(
1+\lambda \right) \left( 1-\lambda \right) ^{3}}{6}+\frac{48\lambda -27}{%
3.2^{6}}\right] \left\vert f^{\prime \prime }(b)\right\vert ^{q}\right) ^{%
\frac{1}{q}}\right\} ,\;\text{for\ }0\leq \lambda \leq \frac{1}{2} \\ 
\frac{\left( b-a\right) ^{2}}{2}\left( \frac{3\lambda -1}{24}\right) ^{1-%
\frac{1}{q}}\left\{ \left( \frac{8\lambda -3}{3.2^{6}}\left\vert f^{\prime
\prime }(a)\right\vert ^{q}+\frac{16\lambda -5}{3.2^{6}}\left\vert f^{\prime
\prime }(b)\right\vert ^{q}\right) ^{\frac{1}{q}}\right. \\ 
\left. +\left( \frac{8\lambda -3}{3.2^{6}}\left\vert f^{\prime \prime
}(b)\right\vert ^{q}+\frac{16\lambda -5}{3.2^{6}}\left\vert f^{\prime \prime
}(a)\right\vert ^{q}\right) ^{\frac{1}{q}}\right\} ,\;\text{for\ }\frac{1}{2}%
\leq \lambda \leq 1%
\end{array}%
\right.
\end{equation*}
\end{theorem}

Let us recall the following special functions:

(1) The Beta function:%
\begin{equation*}
\beta \left( x,y\right) =\frac{\Gamma (x)\Gamma (y)}{\Gamma (x+y)}%
=\dint\limits_{0}^{1}t^{x-1}\left( 1-t\right) ^{y-1}dt,\ \ x,y>0,
\end{equation*}

(2) The incomplete Beta function: 
\begin{equation*}
\beta \left( a,x,y\right) =\dint\limits_{0}^{a}t^{x-1}\left( 1-t\right)
^{y-1}dt,\ \ 0<a<1,\ x,y>0,
\end{equation*}

(3) The hypergeometric function:%
\begin{equation*}
_{2}F_{1}\left( a,b;c;z\right) =\frac{1}{\beta \left( b,c-b\right) }%
\dint\limits_{0}^{1}t^{b-1}\left( 1-t\right) ^{c-b-1}\left( 1-zt\right)
^{-a}dt,\ c>b>0,\ \left\vert z\right\vert <1\text{ (see \cite{AS65}).}
\end{equation*}

The main aim of this article is to establish new generalization of Hermite
Hadamard-type and Simpson-type inequalities for functions whose second
derivatives in absolutely value at certain powers are $(\alpha ,m)-$convex.
To begin with, the author derives a general integral identity for twice
differentiable mappings. By using this integral equality, the author
establishes some new inequalities of the Simpson-like and the
Hermite-Hadamard-like type for these functions.\ 

\section{Main Results}

Let $f:I\subseteq 
\mathbb{R}
\rightarrow 
\mathbb{R}
$ be a differentiable function on $I^{\circ }$, the interior of $I$,
throughout this section we will take%
\begin{eqnarray*}
&&I_{f}\left( x,\lambda ,\kappa ;a,mb\right) \\
&=&\left( 1-\lambda \right) \left[ \frac{\left( x-a\right) ^{\kappa }+\left(
mb-x\right) ^{\kappa }}{mb-a}\right] f(x)+\lambda \left[ \frac{\left(
x-a\right) ^{\kappa }f(a)+\left( mb-x\right) ^{\kappa }f(mb)}{mb-a}\right] \\
&&+\left( \frac{1}{\kappa +1}-\lambda \right) \left[ \frac{\left(
mb-x\right) ^{\kappa +1}-\left( x-a\right) ^{\kappa +1}}{mb-a}\right]
f^{\prime }(x)-\frac{\Gamma \left( \alpha +1\right) }{mb-a}\left[
J_{x^{-}}^{\kappa }f(a)+J_{x^{+}}^{\kappa }f(mb)\right]
\end{eqnarray*}%
where $a,b\in I$ and $m\in \left( 0,1\right] $ with $a<mb$, $\ x\in \lbrack
a,mb]$ , $\lambda \in \left[ 0,1\right] $, $\kappa >0$ and $\Gamma $ is
Euler Gamma function. In order to prove our main results we need the
following identity.

\begin{lemma}
\label{2.1}Let $f:I\subseteq 
\mathbb{R}
\rightarrow 
\mathbb{R}
$ be a twice differentiable function on $I^{\circ }$ such that $f^{\prime
\prime }\in L[a,b]$, where $a,b\in I$ with $a<b$. Then for all $x\in \lbrack
a,mb]$ ,$m\in \left( 0,1\right] $, $\lambda \in \left[ 0,1\right] $ and $%
\alpha >0$ we have:%
\begin{eqnarray}
&&I_{f}\left( x,\lambda ,\kappa ;a,mb\right) =\frac{\left( x-a\right)
^{\kappa +2}}{\left( \kappa +1\right) \left( mb-a\right) }%
\dint\limits_{0}^{1}t\left( \left( \kappa +1\right) \lambda -t^{\kappa
}\right) f^{\prime \prime }\left( tx+\left( 1-t\right) a\right) dt
\label{2-1} \\
&&+\frac{\left( mb-x\right) ^{\kappa +2}}{\left( \kappa +1\right) \left(
mb-a\right) }\dint\limits_{0}^{1}t\left( \left( \kappa +1\right) \lambda
-t^{\kappa }\right) f^{\prime \prime }\left( tx+m\left( 1-t\right) b\right)
dt.  \notag
\end{eqnarray}
\end{lemma}

A simple proof of the equality can be done by performing an integration by
parts in the integrals from the right side and changing the variable. The
details are left to the interested reader.

\begin{theorem}
\label{2.1.1}Let $f:$ $I\subset \lbrack 0,\infty )\rightarrow 
\mathbb{R}
$ be twice differentiable function on $I^{\circ }$ such that $f^{\prime
\prime }\in L[a,b]$, where $a/m,b\in I^{\circ }$ with $a<mb$. If $|f^{\prime
\prime }|^{q}$ is $(\alpha ,m)-$convex on $[a,b]$ for some fixed $\alpha \in %
\left[ 0,1\right] ,$ $m\in \left( 0,1\right] $ and$\ q\geq 1$, then for $%
x\in \lbrack a,mb]$, $\lambda \in \left[ 0,1\right] $ and $\kappa >0$ the
following inequality for fractional integrals holds%
\begin{equation}
\left\vert I_{f}\left( x,\lambda ,\kappa ;a,mb\right) \right\vert
\label{2-2}
\end{equation}%
\begin{eqnarray*}
&\leq &\varphi _{1}^{1-\frac{1}{q}}\left( \kappa ,\lambda \right) \left\{ 
\frac{\left( x-a\right) ^{\kappa +2}}{\left( \kappa +1\right) \left(
mb-a\right) }\left( \left\vert f^{\prime \prime }\left( x\right) \right\vert
^{q}\varphi _{2}\left( \kappa ,\lambda ,\alpha \right) +m\left\vert
f^{\prime \prime }\left( \frac{a}{m}\right) \right\vert ^{q}\varphi
_{3}\left( \kappa ,\lambda ,\alpha \right) \right) ^{\frac{1}{q}}\right. \\
&&+\left. \frac{\left( mb-x\right) ^{\kappa +2}}{\left( \kappa +1\right)
\left( mb-a\right) }\left( \left\vert f^{\prime \prime }\left( x\right)
\right\vert ^{q}\varphi _{2}\left( \kappa ,\lambda ,\alpha \right)
+m\left\vert f^{\prime \prime }\left( b\right) \right\vert ^{q}\varphi
_{3}\left( \kappa ,\lambda ,\alpha \right) \right) ^{\frac{1}{q}}\right\} ,
\end{eqnarray*}%
where 
\begin{eqnarray*}
\varphi _{1}\left( \kappa ,\lambda \right) &=&\left\{ 
\begin{array}{cc}
\frac{\kappa \left[ \left( \kappa +1\right) \lambda \right] ^{\frac{\kappa +2%
}{\kappa }}}{\kappa +2}-\frac{\left( \kappa +1\right) \lambda }{2}+\frac{1}{%
\kappa +2}, & 0\leq \lambda \leq \frac{1}{\kappa +1} \\ 
\frac{\left( \kappa +1\right) \lambda }{2}-\frac{1}{\kappa +2}, & \frac{1}{%
\kappa +1}<\lambda \leq 1%
\end{array}%
\right. , \\
\varphi _{2}\left( \kappa ,\lambda ,\alpha \right) &=&\left\{ 
\begin{array}{cc}
\frac{2\kappa \left[ \left( \kappa +1\right) \lambda \right] ^{\frac{\kappa
+\alpha +2}{\kappa }}}{\left( \alpha +2\right) \left( \kappa +\alpha
+2\right) }-\frac{\left( \kappa +1\right) \lambda }{\alpha +2}+\frac{1}{%
\kappa +\alpha +2}, & 0\leq \lambda \leq \frac{1}{\kappa +1} \\ 
\frac{\left( \kappa +1\right) \lambda }{\alpha +2}-\frac{1}{\kappa +\alpha +2%
}, & \frac{1}{\kappa +1}<\lambda \leq 1%
\end{array}%
\right. , \\
\varphi _{3}\left( \kappa ,\lambda ,\alpha \right) &=&\left\{ 
\begin{array}{cc}
\frac{\kappa \left[ \left( \kappa +1\right) \lambda \right] ^{\frac{\kappa +2%
}{\kappa }}}{\left( \kappa +2\right) }-\frac{2\kappa \left[ \left( \kappa
+1\right) \lambda \right] ^{\frac{\kappa +\alpha +2}{\kappa }}}{\left(
\alpha +2\right) \left( \kappa +\alpha +2\right) }-\frac{\alpha \left(
\kappa +1\right) \lambda }{2(\alpha +2)}+\frac{\kappa }{\left( \kappa
+2\right) \left( \kappa +\alpha +2\right) }, & 0\leq \lambda \leq \frac{1}{%
\kappa +1} \\ 
\frac{\alpha \left( \kappa +1\right) \lambda }{2(\alpha +2)}-\frac{\kappa }{%
\left( \kappa +2\right) \left( \kappa +\alpha +2\right) }, & \frac{1}{\kappa
+1}<\lambda \leq 1%
\end{array}%
\right. .
\end{eqnarray*}
\end{theorem}

\begin{proof}
From Lemma \ref{2.1}, using the property of the modulus and the power-mean
inequality we have%
\begin{eqnarray}
&&\left\vert I_{f}\left( x,\lambda ,\kappa ,a,mb\right) \right\vert \leq 
\frac{\left( x-a\right) ^{\kappa +2}}{\left( \kappa +1\right) \left(
mb-a\right) }\dint\limits_{0}^{1}t\left\vert \left( \kappa +1\right) \lambda
-t^{\kappa }\right\vert \left\vert f^{\prime \prime }\left( tx+\left(
1-t\right) a\right) \right\vert dt  \notag \\
&&+\frac{\left( mb-x\right) ^{\kappa +2}}{\left( \kappa +1\right) \left(
mb-a\right) }\dint\limits_{0}^{1}t\left\vert \left( \kappa +1\right) \lambda
-t^{\kappa }\right\vert \left\vert f^{\prime \prime }\left( tx+m\left(
1-t\right) b\right) \right\vert dt  \notag \\
&\leq &\frac{\left( x-a\right) ^{\kappa +2}}{\left( \kappa +1\right) \left(
mb-a\right) }\left( \dint\limits_{0}^{1}t\left\vert \left( \kappa +1\right)
\lambda -t^{\kappa }\right\vert dt\right) ^{1-\frac{1}{q}}  \notag \\
&&\times \left( \dint\limits_{0}^{1}t\left\vert \left( \kappa +1\right)
\lambda -t^{\kappa }\right\vert \left\vert f^{\prime \prime }\left(
tx+\left( 1-t\right) a\right) \right\vert ^{q}dt\right) ^{\frac{1}{q}} 
\notag \\
&&+\frac{\left( mb-x\right) ^{\kappa +2}}{\left( \kappa +1\right) \left(
mb-a\right) }\left( \dint\limits_{0}^{1}t\left\vert \left( \kappa +1\right)
\lambda -t^{\kappa }\right\vert dt\right) ^{1-\frac{1}{q}}  \notag \\
&&\times \left( \dint\limits_{0}^{1}t\left\vert \left( \kappa +1\right)
\lambda -t^{\kappa }\right\vert \left\vert f^{\prime \prime }\left(
tx+\left( 1-t\right) mb\right) \right\vert ^{q}dt\right) ^{\frac{1}{q}}.
\label{2-2a}
\end{eqnarray}%
Since $\left\vert f^{\prime \prime }\right\vert ^{q}$ is $\left( \alpha
,m\right) -$convex on $[a,b]$ we get%
\begin{eqnarray}
\dint\limits_{0}^{1}t\left\vert \left( \kappa +1\right) \lambda -t^{\kappa
}\right\vert \left\vert f^{\prime \prime }\left( tx+\left( 1-t\right)
a\right) \right\vert ^{q}dt &\leq &\dint\limits_{0}^{1}t\left\vert \left(
\kappa +1\right) \lambda -t^{\kappa }\right\vert \left( t^{\alpha
}\left\vert f^{\prime \prime }\left( x\right) \right\vert ^{q}+m\left(
1-t\alpha \right) \left\vert f^{\prime \prime }\left( \frac{a}{m}\right)
\right\vert ^{q}\right) dt  \notag \\
&=&\left\vert f^{\prime \prime }\left( x\right) \right\vert ^{q}\varphi
_{2}\left( \kappa ,\lambda ,\alpha \right) +m\left\vert f^{\prime \prime
}\left( \frac{a}{m}\right) \right\vert ^{q}\varphi _{3}\left( \kappa
,\lambda ,\alpha \right) ,  \label{2-2b}
\end{eqnarray}%
\begin{eqnarray}
\dint\limits_{0}^{1}t\left\vert \left( \kappa +1\right) \lambda -t^{\kappa
}\right\vert \left\vert f^{\prime \prime }\left( tx+m\left( 1-t\right)
b\right) \right\vert ^{q}dt &\leq &\dint\limits_{0}^{1}t\left\vert \left(
\kappa +1\right) \lambda -t^{\kappa }\right\vert \left( t^{\alpha
}\left\vert f^{\prime \prime }\left( x\right) \right\vert ^{q}+m\left(
1-t^{\alpha }\right) \left\vert f^{\prime \prime }\left( b\right)
\right\vert ^{q}\right) dt  \notag \\
&=&\left\vert f^{\prime \prime }\left( x\right) \right\vert ^{q}\varphi
_{2}\left( \kappa ,\lambda ,\alpha \right) +m\left\vert f^{\prime \prime
}\left( b\right) \right\vert ^{q}\varphi _{3}\left( \kappa ,\lambda ,\alpha
\right) ,  \label{2-2c}
\end{eqnarray}%
where we use the fact that%
\begin{eqnarray*}
&&\varphi _{3}\left( \kappa ,\lambda ,\alpha \right) \\
&=&\dint\limits_{0}^{1}t\left\vert \left( \kappa +1\right) \lambda
-t^{\kappa }\right\vert \left( 1-t^{\alpha }\right) dt \\
&=&\left\{ 
\begin{array}{cc}
\begin{array}{c}
\left( \kappa +1\right) \lambda \dint\limits_{0}^{\left[ \left( \kappa
+1\right) \lambda \right] ^{\frac{1}{\kappa }}}t\left( 1-t^{\alpha }\right)
dt-\dint\limits_{0}^{\left[ \left( \kappa +1\right) \lambda \right] ^{\frac{1%
}{\kappa }}}t^{\kappa +1}\left( 1-t^{\alpha }\right) dt \\ 
-\left( \kappa +1\right) \lambda \dint\limits_{\left[ \left( \kappa
+1\right) \lambda \right] ^{\frac{1}{\kappa }}}^{1}t\left( 1-t^{\alpha
}\right) dt+\dint\limits_{\left[ \left( \kappa +1\right) \lambda \right] ^{%
\frac{1}{\kappa }}}^{1}t^{\kappa +1}\left( 1-t^{\alpha }\right) dt%
\end{array}%
, & 0\leq \lambda \leq \frac{1}{\kappa +1} \\ 
\left( \kappa +1\right) \lambda \dint\limits_{0}^{1}t\left( 1-t^{\alpha
}\right) dt-\dint\limits_{0}^{1}t^{\kappa +1}\left( 1-t^{\alpha }\right) dt,
& \frac{1}{\kappa +1}<\lambda \leq 1%
\end{array}%
\right. \\
&=&\left\{ 
\begin{array}{cc}
\frac{\kappa \left[ \left( \kappa +1\right) \lambda \right] ^{\frac{\kappa +2%
}{\kappa }}}{\left( \kappa +2\right) }-\frac{2\kappa \left[ \left( \kappa
+1\right) \lambda \right] ^{\frac{\kappa +\alpha +2}{\kappa }}}{\left(
\alpha +2\right) \left( \kappa +\alpha +2\right) }-\frac{\alpha \left(
\kappa +1\right) \lambda }{2(\alpha +2)}+\frac{\kappa }{\left( \kappa
+2\right) \left( \kappa +\alpha +2\right) }, & 0\leq \lambda \leq \frac{1}{%
\kappa +1} \\ 
\frac{\alpha \left( \kappa +1\right) \lambda }{2(\alpha +2)}-\frac{\kappa }{%
\left( \kappa +2\right) \left( \kappa +\alpha +2\right) }, & \frac{1}{\kappa
+1}<\lambda \leq 1%
\end{array}%
\right. ,
\end{eqnarray*}%
\begin{eqnarray*}
&&\varphi _{2}\left( \kappa ,\lambda ,\alpha \right) \\
&=&\dint\limits_{0}^{1}t\left\vert \left( \kappa +1\right) \lambda
-t^{\kappa }\right\vert t^{\alpha }dt \\
&=&\left\{ 
\begin{array}{cc}
\frac{2\kappa \left[ \left( \kappa +1\right) \lambda \right] ^{\frac{\kappa
+\alpha +2}{\kappa }}}{\left( \alpha +2\right) \left( \kappa +\alpha
+2\right) }-\frac{\left( \kappa +1\right) \lambda }{\alpha +2}+\frac{1}{%
\kappa +\alpha +2}, & 0\leq \lambda \leq \frac{1}{\kappa +1} \\ 
\frac{\left( \kappa +1\right) \lambda }{\alpha +2}-\frac{1}{\kappa +\alpha +2%
}, & \frac{1}{\kappa +1}<\lambda \leq 1%
\end{array}%
\right.
\end{eqnarray*}%
and 
\begin{eqnarray}
\varphi _{1}\left( \kappa ,\lambda \right)
&=&\dint\limits_{0}^{1}t\left\vert \left( \kappa +1\right) \lambda
-t^{\kappa }\right\vert dt  \notag \\
&=&\left\{ 
\begin{array}{cc}
\frac{\kappa \left[ \left( \kappa +1\right) \lambda \right] ^{1+\frac{2}{%
\kappa }}}{\kappa +2}-\frac{\left( \kappa +1\right) \lambda }{2}+\frac{1}{%
\kappa +2}, & 0\leq \lambda \leq \frac{1}{\kappa +1} \\ 
\frac{\left( \kappa +1\right) \lambda }{2}-\frac{1}{\kappa +2}, & \frac{1}{%
\kappa +1}<\lambda \leq 1%
\end{array}%
\right.  \label{2-2d}
\end{eqnarray}%
Hence, If we use (\ref{2-2b}), (\ref{2-2c}) and (\ref{2-2d}) in (\ref{2-2a}%
), we obtain the desired result. This completes the proof.
\end{proof}

\begin{corollary}
\label{2.a}In Theorem \ref{2.1.1},

(a) If we choose $q=1,$ then we get:%
\begin{eqnarray*}
\left\vert I_{f}\left( x,\lambda ,\kappa ;a,mb\right) \right\vert &\leq
&\left\{ \frac{\left( x-a\right) ^{\kappa +1}}{mb-a}\left( \left\vert
f^{\prime \prime }\left( x\right) \right\vert \varphi _{2}\left( \kappa
,\lambda ,\alpha \right) +m\left\vert f^{\prime \prime }\left( \frac{a}{m}%
\right) \right\vert \varphi _{3}\left( \kappa ,\lambda ,\alpha \right)
\right) \right. \\
&&+\left. \frac{\left( mb-x\right) ^{\kappa +1}}{mb-a}\left( \left\vert
f^{\prime \prime }\left( x\right) \right\vert \varphi _{2}\left( \kappa
,\lambda ,\alpha \right) +m\left\vert f^{\prime \prime }\left( b\right)
\right\vert \varphi _{3}\left( \kappa ,\lambda ,\alpha \right) \right)
\right\} .
\end{eqnarray*}

(b) If we choose $x=\frac{a+mb}{2},$ then we get:%
\begin{eqnarray}
&&\left\vert \frac{2^{\kappa -1}}{\left( mb-a\right) ^{\kappa -1}}%
I_{f}\left( \frac{a+mb}{2},\lambda ,\kappa ;a,mb\right) \right\vert  \notag
\\
&=&\left\vert \left( 1-\lambda \right) f\left( \frac{a+mb}{2}\right)
+\lambda \left( \frac{f(a)+f(mb)}{2}\right) -\frac{\Gamma \left( \kappa
+1\right) 2^{\kappa -1}}{\left( mb-a\right) ^{\kappa }}\left[ J_{\left( 
\frac{a+mb}{2}\right) ^{-}}^{\kappa }f(a)+J_{\left( \frac{a+mb}{2}\right)
^{+}}^{\kappa }f(mb)\right] \right\vert  \notag
\end{eqnarray}%
\begin{equation*}
\leq \frac{\left( mb-a\right) ^{2}}{8\left( \kappa +1\right) }\varphi
_{1}^{1-\frac{1}{q}}\left( \kappa ,\lambda \right) \left\{ \left( \left\vert
f^{\prime \prime }\left( \frac{a+mb}{2}\right) \right\vert ^{q}\varphi
_{2}\left( \kappa ,\lambda ,\alpha \right) +m\left\vert f^{\prime \prime
}\left( \frac{a}{m}\right) \right\vert ^{q}\varphi _{3}\left( \kappa
,\lambda ,\alpha \right) \right) ^{\frac{1}{q}}\right.
\end{equation*}%
\begin{equation}
\left. +\left( \left\vert f^{\prime \prime }\left( \frac{a+mb}{2}\right)
\right\vert ^{q}\varphi _{2}\left( \kappa ,\lambda ,\alpha \right)
+m\left\vert f^{\prime \prime }\left( b\right) \right\vert ^{q}\varphi
_{3}\left( \kappa ,\lambda ,\alpha \right) \right) ^{\frac{1}{q}}\right\} .
\label{2-2e}
\end{equation}

(c) If we choose $x=\frac{a+mb}{2}$ and $\lambda =\frac{1}{3},$ then we get:%
\begin{eqnarray*}
&&\left\vert \frac{2^{\kappa -1}}{\left( mb-a\right) ^{\kappa -1}}%
I_{f}\left( \frac{a+mb}{2},\frac{1}{3},\kappa ;a,mb\right) \right\vert \\
&=&\left\vert \frac{1}{6}\left[ f(a)+4f\left( \frac{a+mb}{2}\right) +f(mb)%
\right] -\frac{\Gamma \left( \kappa +1\right) 2^{\kappa -1}}{\left(
mb-a\right) ^{\kappa }}\left[ J_{\left( \frac{a+b}{2}\right) ^{-}}^{\kappa
}f(a)+J_{\left( \frac{a+b}{2}\right) ^{+}}^{\kappa }f(mb)\right] \right\vert
\\
&\leq &\frac{\left( mb-a\right) ^{2}}{8\left( \kappa +1\right) }\varphi
_{1}^{1-\frac{1}{q}}\left( \kappa ,\frac{1}{3}\right) \left\{ \left(
\left\vert f^{\prime \prime }\left( \frac{a+mb}{2}\right) \right\vert
^{q}\varphi _{2}\left( \kappa ,\frac{1}{3},\alpha \right) +m\left\vert
f^{\prime \prime }\left( \frac{a}{m}\right) \right\vert ^{q}\varphi
_{3}\left( \kappa ,\frac{1}{3},\alpha \right) \right) ^{\frac{1}{q}}\right.
\\
&&\left. +\left( \left\vert f^{\prime \prime }\left( \frac{a+mb}{2}\right)
\right\vert ^{q}\varphi _{2}\left( \kappa ,\frac{1}{3},\alpha \right)
+m\left\vert f^{\prime \prime }\left( b\right) \right\vert ^{q}\varphi
_{3}\left( \kappa ,\frac{1}{3},\alpha \right) \right) ^{\frac{1}{q}}\right\}
.
\end{eqnarray*}

(d) If we choose $x=\frac{a+mb}{2},\ \lambda =\frac{1}{3},$ and $\kappa =1$,
then we get:%
\begin{eqnarray*}
&&\left\vert I_{f}\left( \frac{a+mb}{2},\frac{1}{3},1,a,mb\right) \right\vert
\\
&=&\left\vert \frac{1}{6}\left[ f(a)+4f\left( \frac{a+mb}{2}\right) +f(mb)%
\right] -\frac{1}{\left( mb-a\right) }\dint\limits_{a}^{b}f(x)dx\right\vert
\\
&\leq &\frac{\left( mb-a\right) ^{2}}{162}\left( \frac{81}{8}\right) ^{\frac{%
1}{q}}\left\{ \left( \left\vert f^{\prime \prime }\left( \frac{a+mb}{2}%
\right) \right\vert ^{q}\varphi _{2}\left( 1,\frac{1}{3},\alpha \right)
+m\left\vert f^{\prime \prime }\left( \frac{a}{m}\right) \right\vert
^{q}\varphi _{3}\left( 1,\frac{1}{3},\alpha \right) \right) ^{\frac{1}{q}%
}\right. \\
&&\left. +\left( \left\vert f^{\prime \prime }\left( \frac{a+mb}{2}\right)
\right\vert ^{q}\varphi _{2}\left( 1,\frac{1}{3},\alpha \right) +m\left\vert
f^{\prime \prime }\left( b\right) \right\vert ^{q}\varphi _{3}\left( 1,\frac{%
1}{3},\alpha \right) \right) ^{\frac{1}{q}}\right\} .
\end{eqnarray*}%
where%
\begin{equation*}
\varphi _{2}\left( 1,\frac{1}{3},\alpha \right) =\frac{2^{\alpha
+4}-23^{\alpha +2}+3^{\alpha +3}(\alpha +2)}{3^{\alpha +3}(\alpha +2)(\alpha
+3)},
\end{equation*}%
\begin{equation*}
\varphi _{3}\left( 1,\frac{1}{3},\alpha \right) =\frac{-2^{\alpha +4}-\alpha
3^{\alpha +2}(\alpha +3)+3^{\alpha +3}(\alpha +2)+83^{\alpha -1}(\alpha
+2)(\alpha +3)}{3^{\alpha +3}(\alpha +2)(\alpha +3)}.
\end{equation*}

(e) If we choose $x=\frac{a+mb}{2}$ and$\ \lambda =0,$ then we get:%
\begin{eqnarray*}
&&\left\vert \frac{2^{\kappa -1}}{\left( mb-a\right) ^{\kappa -1}}%
I_{f}\left( \frac{a+mb}{2},0,\kappa ;a,mb\right) \right\vert \\
&=&\left\vert f\left( \frac{a+mb}{2}\right) -\frac{\Gamma \left( \kappa
+1\right) 2^{\kappa -1}}{\left( mb-a\right) ^{\kappa }}\left[ J_{\left( 
\frac{a+mb}{2}\right) ^{-}}^{\kappa }f(a)+J_{\left( \frac{a+mb}{2}\right)
^{+}}^{\kappa }f(b)\right] \right\vert \\
&\leq &\frac{\left( mb-a\right) ^{2}}{8\left( \kappa +1\right) (\alpha
\kappa +2)}\left( \frac{\kappa +2}{\kappa +s+2}\right) ^{\frac{1}{q}}\left\{ %
\left[ \left\vert f^{\prime \prime }\left( \frac{a+mb}{2}\right) \right\vert
^{q}+\frac{\kappa m\left\vert f^{\prime \prime }\left( \frac{a}{m}\right)
\right\vert ^{q}}{(\kappa +2)}\right] ^{\frac{1}{q}}\right. \\
&&\left. +\left[ \left\vert f^{\prime \prime }\left( \frac{a+mb}{2}\right)
\right\vert ^{q}+\frac{\kappa m\left\vert f^{\prime \prime }\left( b\right)
\right\vert ^{q}}{(\kappa +2)}\right] ^{\frac{1}{q}}\right\} .
\end{eqnarray*}

(f) If we choose $x=\frac{a+mb}{2},\ \lambda =0,$ and $\kappa =1$, then we
get: 
\begin{eqnarray*}
&&\left\vert I_{f}\left( \frac{a+mb}{2},0,1;a,mb\right) \right\vert \\
&=&\left\vert f\left( \frac{a+mb}{2}\right) -\frac{1}{\left( mb-a\right) }%
\dint\limits_{a}^{mb}f(x)dx\right\vert \\
&\leq &\frac{\left( mb-a\right) ^{2}}{48}\left( \frac{1}{s+3}\right) ^{\frac{%
1}{q}}\left\{ \left[ 3\left\vert f^{\prime \prime }\left( \frac{a+mb}{2}%
\right) \right\vert ^{q}+m\left\vert f^{\prime \prime }\left( \frac{a}{m}%
\right) \right\vert ^{q}\right] ^{\frac{1}{q}}\right. \\
&&\left. +\left[ 3\left\vert f^{\prime \prime }\left( \frac{a+mb}{2}\right)
\right\vert ^{q}+m\left\vert f^{\prime \prime }\left( b\right) \right\vert
^{q}\right] ^{\frac{1}{q}}\right\} .
\end{eqnarray*}

(g) If we choose$\ x=\frac{a+mb}{2}$ and $\lambda =1,$ then we get:%
\begin{eqnarray*}
&&\left\vert \frac{2^{\kappa -1}}{\left( mb-a\right) ^{\kappa -1}}%
I_{f}\left( \frac{a+mb}{2},1,\kappa ;a,mb\right) \right\vert \\
&=&\left\vert \frac{f(a)+f(mb)}{2}-\frac{\Gamma \left( \kappa +1\right)
2^{\kappa -1}}{\left( mb-a\right) ^{\kappa }}\left[ J_{\left( \frac{a+mb}{2}%
\right) ^{-}}^{\kappa }f(a)+J_{\left( \frac{a+mb}{2}\right) ^{+}}^{\kappa
}f(mb)\right] \right\vert \\
&\leq &\frac{\left( mb-a\right) ^{2}}{8\left( \kappa +1\right) }\left( \frac{%
\kappa \left( \kappa +3\right) }{2\left( \kappa +2\right) }\right) ^{1-\frac{%
1}{q}} \\
&&\times \left\{ \left[ \frac{\kappa \left( \kappa +\alpha +3\right)
\left\vert f^{\prime \prime }\left( \frac{a+mb}{2}\right) \right\vert ^{q}}{%
\left( \alpha +2\right) \left( \kappa +\alpha +2\right) }+m\left\vert
f^{\prime \prime }\left( \frac{a}{m}\right) \right\vert ^{q}\left( \frac{%
\alpha (\kappa +1)}{2(\alpha +2)}-\frac{\kappa }{(\kappa +2)(\kappa +\alpha
+2)}\right) \right] ^{\frac{1}{q}}\right. \\
&&\left. +\left[ \frac{\kappa \left( \kappa +\alpha +3\right) \left\vert
f^{\prime \prime }\left( \frac{a+mb}{2}\right) \right\vert ^{q}}{\left(
\alpha +2\right) \left( \kappa +\alpha +2\right) }+m\left\vert f^{\prime
\prime }\left( b\right) \right\vert ^{q}\left( \frac{\alpha (\kappa +1)}{%
2(\alpha +2)}-\frac{\kappa }{(\kappa +2)(\kappa +\alpha +2)}\right) \right]
^{\frac{1}{q}}\right\} .
\end{eqnarray*}

(h) If we choose$\ x=\frac{a+mb}{2},\lambda =1$ and $\kappa =1$, then we get:%
\begin{eqnarray*}
&&\left\vert I_{f}\left( \frac{a+b}{2},1,1;a,mb\right) \right\vert \\
&=&\left\vert \frac{f(a)+f(mb)}{2}-\frac{1}{\left( mb-a\right) }%
\dint\limits_{a}^{mb}f(x)dx\right\vert \\
&\leq &\frac{\left( mb-a\right) ^{2}}{16}\left( \frac{2}{3}\right) ^{1-\frac{%
1}{q}} \\
&&\times \left\{ \left[ \frac{\left( \alpha +4\right) \left\vert f^{\prime
\prime }\left( \frac{a+mb}{2}\right) \right\vert ^{q}}{\left( \alpha
+2\right) \left( \alpha +3\right) }+m\left\vert f^{\prime \prime }\left( 
\frac{a}{m}\right) \right\vert ^{q}\frac{3\alpha ^{2}+8\alpha -2}{3(\alpha
+2)(\alpha +3)}\right] ^{\frac{1}{q}}\right. \\
&&\left. +\left[ \frac{\left( \alpha +4\right) \left\vert f^{\prime \prime
}\left( \frac{a+mb}{2}\right) \right\vert ^{q}}{\left( \alpha +2\right)
\left( \alpha +3\right) }+m\left\vert f^{\prime \prime }\left( b\right)
\right\vert ^{q}\frac{3\alpha ^{2}+8\alpha -2}{3(\alpha +2)(\alpha +3)}%
\right] ^{\frac{1}{q}}\right\} .
\end{eqnarray*}
\end{corollary}

\begin{remark}
In \ (b) of Corollary \ref{2.a}, if we choose $\kappa =m=\alpha =1,$ then
the inequality (\ref{2-2e}) reduces to the following inequality which is
better than the inequality (\ref{1-1.1})%
\begin{equation*}
\left\vert \left( 1-\lambda \right) f\left( \frac{a+b}{2}\right) +\lambda
\left( \frac{f(a)+f(b)}{2}\right) -\frac{1}{\left( b-a\right) }%
\dint\limits_{a}^{b}f(x)dx\right\vert
\end{equation*}%
\begin{eqnarray*}
&\leq &\frac{\left( b-a\right) ^{2}}{16}\varphi _{1}^{1-\frac{1}{q}}\left(
1,\lambda \right) \left\{ \left( \left\vert f^{\prime \prime }\left( \frac{%
a+b}{2}\right) \right\vert ^{q}\varphi _{2}\left( 1,\lambda ,1\right)
+\left\vert f^{\prime \prime }\left( a\right) \right\vert ^{q}\varphi
_{3}\left( 1,\lambda ,1\right) \right) ^{\frac{1}{q}}\right. \\
&&+\left. \left( \left\vert f^{\prime \prime }\left( \frac{a+b}{2}\right)
\right\vert ^{q}\varphi _{2}\left( 1,\lambda ,1\right) +\left\vert f^{\prime
\prime }\left( b\right) \right\vert ^{q}\varphi _{3}\left( 1,\lambda
,1\right) \right) ^{\frac{1}{q}}\right\} ,
\end{eqnarray*}%
where%
\begin{equation*}
\varphi _{1}\left( 1,\lambda \right) =\left\{ 
\begin{array}{cc}
8\left( \frac{\lambda ^{3}}{3}+\frac{1-3\lambda }{24}\right) , & \;0\leq
\lambda \leq \frac{1}{2} \\ 
\frac{3\lambda -1}{3}, & \frac{1}{2}<\lambda \leq 1%
\end{array}%
\right. ,
\end{equation*}%
\begin{equation*}
\varphi _{2}\left( 1,\lambda ,1\right) =\left\{ 
\begin{array}{cc}
16\left( \frac{\lambda ^{4}}{6}+\frac{3-8\lambda }{3.2^{6}}\right) , & 0\leq
\lambda \leq \frac{1}{2} \\ 
\frac{8\lambda -3}{12}, & \frac{1}{2}<\lambda \leq 1%
\end{array}%
\right. ,
\end{equation*}%
\begin{equation*}
\varphi _{3}\left( 1,\lambda ,1\right) =\left\{ 
\begin{array}{cc}
\frac{-8\lambda ^{4}+8\lambda ^{3}-\lambda }{3}+\frac{1}{12}, & 0\leq
\lambda \leq \frac{1}{2} \\ 
\frac{4\lambda -1}{12}, & \frac{1}{2}<\lambda \leq 1%
\end{array}%
\right. .
\end{equation*}
\end{remark}

\begin{theorem}
\label{2.2}Let $f:$ $I\subset \lbrack 0,\infty )\rightarrow 
\mathbb{R}
$ be twice differentiable function on $I^{\circ }$ such that $f^{\prime
\prime }\in L[a,b]$, where $a/m,b\in I^{\circ }$ with $a<mb$. If $|f^{\prime
\prime }|^{q}$ is $(\alpha ,m)-$convex on $[a,b]$ for some fixed $\alpha \in %
\left[ 0,1\right] ,\;m\in \left( 0,1\right] $ and $q>1$, then for $x\in
\lbrack a,mb]$, $\lambda \in \left[ 0,1\right] $ and $\kappa >0$ the
following inequality for fractional integrals holds%
\begin{eqnarray}
&&\left\vert I_{f}\left( x,\lambda ,\kappa ;a,mb\right) \right\vert
\label{2-3} \\
&\leq &\varphi _{4}^{\frac{1}{p}}\left( \kappa ,\lambda ,p\right) \left\{ 
\frac{\left( x-a\right) ^{\kappa +2}}{\left( \kappa +1\right) \left(
mb-a\right) }\left( \frac{\left\vert f^{\prime \prime }\left( x\right)
\right\vert ^{q}+\alpha m\left\vert f^{\prime \prime }\left( \frac{a}{m}%
\right) \right\vert ^{q}}{\alpha +1}\right) ^{\frac{1}{q}}\right.  \notag \\
&&+\left. \frac{\left( mb-x\right) ^{\kappa +2}}{\left( \kappa +1\right)
\left( mb-a\right) }\left( \frac{\left\vert f^{\prime \prime }\left(
x\right) \right\vert ^{q}+\alpha m\left\vert f^{\prime \prime }\left(
b\right) \right\vert ^{q}}{\alpha +1}\right) ^{\frac{1}{q}}\right\} ,  \notag
\end{eqnarray}%
where $p=\frac{q}{q-1}$ and 
\begin{eqnarray*}
&&\varphi _{4}\left( \kappa ,\lambda ,p\right) \\
&=&\left\{ 
\begin{array}{cc}
\frac{1}{p\left( \kappa +1\right) +1}, & \lambda =0 \\ 
\left[ 
\begin{array}{c}
\frac{\left[ \left( \kappa +1\right) \lambda \right] ^{\frac{1+\left( \kappa
+1\right) p}{\kappa }}}{\kappa }\beta \left( \frac{1+p}{\kappa },1+p\right)
\\ 
+\frac{\left[ 1-\left( \kappa +1\right) \lambda \right] ^{p+1}}{p+1}%
._{2}F_{1}\left( 1-\frac{1+p}{\kappa },1;p+2;1-\left( \kappa +1\right)
\lambda \right)%
\end{array}%
\right] , & 0<\lambda <\frac{1}{\kappa +1} \\ 
\frac{\left[ \left( \kappa +1\right) \lambda \right] ^{\frac{p\left( \kappa
+1\right) +1}{\kappa }}}{\kappa }\beta \left( \frac{1}{\left( \kappa
+1\right) \lambda };\frac{1+p}{\kappa },1+p\right) , & \frac{1}{\kappa +1}%
\leq \lambda \leq 1%
\end{array}%
\right. .
\end{eqnarray*}
\end{theorem}

\begin{proof}
From Lemma \ref{2.1}, property of the modulus and using the H\"{o}lder
inequality we have%
\begin{eqnarray*}
&&\left\vert I_{f}\left( x,\lambda ,\kappa ;a,mb\right) \right\vert \leq 
\frac{\left( x-a\right) ^{\kappa +2}}{\left( \kappa +1\right) \left(
mb-a\right) }\dint\limits_{0}^{1}t\left\vert \left( \kappa +1\right) \lambda
-t^{\kappa }\right\vert \left\vert f^{\prime \prime }\left( tx+\left(
1-t\right) a\right) \right\vert dt \\
&&+\frac{\left( mb-x\right) ^{\kappa +2}}{\left( \kappa +1\right) \left(
mb-a\right) }\dint\limits_{0}^{1}t\left\vert \left( \kappa +1\right) \lambda
-t^{\kappa }\right\vert \left\vert f^{\prime \prime }\left( tx+m\left(
1-t\right) b\right) \right\vert dt \\
&\leq &\frac{\left( x-a\right) ^{\kappa +2}}{\left( \kappa +1\right) \left(
mb-a\right) }\left( \dint\limits_{0}^{1}t^{p}\left\vert \left( \kappa
+1\right) \lambda -t^{\kappa }\right\vert ^{p}dt\right) ^{\frac{1}{p}}\left(
\dint\limits_{0}^{1}\left\vert f^{\prime \prime }\left( tx+\left( 1-t\right)
a\right) \right\vert ^{q}dt\right) ^{\frac{1}{q}}
\end{eqnarray*}%
\begin{equation}
+\frac{\left( mb-x\right) ^{\kappa +2}}{\left( \kappa +1\right) \left(
mb-a\right) }\left( \dint\limits_{0}^{1}t^{p}\left\vert \left( \kappa
+1\right) \lambda -t^{\kappa }\right\vert ^{p}dt\right) ^{\frac{1}{p}}\left(
\dint\limits_{0}^{1}\left\vert f^{\prime \prime }\left( tx+m\left(
1-t\right) b\right) \right\vert ^{q}dt\right) ^{\frac{1}{q}}  \label{2-3a}
\end{equation}%
Since $\left\vert f^{\prime \prime }\right\vert ^{q}$ is $\left( \alpha
,m\right) -$convex on $[a,b]$ we get%
\begin{eqnarray}
\dint\limits_{0}^{1}\left\vert f^{\prime \prime }\left( tx+\left( 1-t\right)
a\right) \right\vert ^{q}dt &\leq &\dint\limits_{0}^{1}\left( t^{\alpha
}\left\vert f^{\prime \prime }\left( x\right) \right\vert ^{q}+m\left(
1-t^{\alpha }\right) \left\vert f^{\prime \prime }\left( \frac{a}{m}\right)
\right\vert ^{q}\right) dt  \notag \\
&=&\frac{\left\vert f^{\prime \prime }\left( x\right) \right\vert
^{q}+m\alpha \left\vert f^{\prime \prime }\left( \frac{a}{m}\right)
\right\vert ^{q}}{\alpha +1},  \label{2-3b}
\end{eqnarray}%
\begin{eqnarray}
\dint\limits_{0}^{1}\left\vert f^{\prime \prime }\left( tx+m\left(
1-t\right) b\right) \right\vert ^{q}dt &\leq &\dint\limits_{0}^{1}\left(
t^{\alpha }\left\vert f^{\prime \prime }\left( x\right) \right\vert
^{q}+m\left( 1-t^{\alpha }\right) \left\vert f^{\prime \prime }\left(
b\right) \right\vert ^{q}\right) dt  \notag \\
&=&\frac{\left\vert f^{\prime \prime }\left( x\right) \right\vert
^{q}+m\alpha \left\vert f^{\prime \prime }\left( b\right) \right\vert ^{q}}{%
\alpha +1},  \label{2-3c}
\end{eqnarray}%
and%
\begin{equation}
\dint\limits_{0}^{1}t^{p}\left\vert \left( \alpha +1\right) \lambda
-t^{\kappa }\right\vert ^{p}dt  \label{2-3d}
\end{equation}%
\begin{eqnarray*}
&=&\left\{ 
\begin{array}{cc}
\dint\limits_{0}^{1}t^{\left( \kappa +1\right) p}dt & \lambda =0 \\ 
\dint\limits_{0}^{\left[ \left( \kappa +1\right) \lambda \right] ^{\frac{1}{%
\alpha }}}t^{p}\left[ \left( \kappa +1\right) \lambda -t^{\kappa }\right]
^{p}dt+\dint\limits_{\left[ \left( \kappa +1\right) \lambda \right] ^{\frac{1%
}{\alpha }}}^{1}t^{p}\left[ t^{\kappa }-\left( \kappa +1\right) \lambda %
\right] ^{p}dt, & 0<\lambda <\frac{1}{\kappa +1} \\ 
\dint\limits_{0}^{1}t^{p}\left[ \left( \kappa +1\right) \lambda -t^{\kappa }%
\right] ^{p}dt, & \frac{1}{\kappa +1}\leq \lambda \leq 1%
\end{array}%
\right. \\
&=&\left\{ 
\begin{array}{cc}
\frac{1}{p\left( \kappa +1\right) +1}, & \lambda =0 \\ 
\left[ 
\begin{array}{c}
\frac{\left[ \left( \kappa +1\right) \lambda \right] ^{\frac{1+\left( \kappa
+1\right) p}{\kappa }}}{\kappa }\beta \left( \frac{1+p}{\kappa },1+p\right)
\\ 
+\frac{\left[ 1-\left( \kappa +1\right) \lambda \right] ^{p+1}}{p+1}%
._{2}F_{1}\left( 1-\frac{1+p}{\kappa },1;p+2;1-\left( \kappa +1\right)
\lambda \right)%
\end{array}%
\right] , & 0<\lambda <\frac{1}{\kappa +1} \\ 
\frac{\left[ \left( \kappa +1\right) \lambda \right] ^{\frac{1+\left( \kappa
+1\right) p}{\kappa }}}{\kappa }\beta \left( \frac{1}{\left( \kappa
+1\right) \lambda };\frac{1+p}{\kappa },1+p\right) , & \frac{1}{\kappa +1}%
\leq \lambda \leq 1%
\end{array}%
\right.
\end{eqnarray*}%
Hence, If we use (\ref{2-3b}), (\ref{2-3c}) and (\ref{2-3d}) in (\ref{2-3a}%
), we obtain the desired result. This completes the proof.
\end{proof}

\begin{corollary}
\label{2.2b}In Theorem \ref{2.2}

(a) If we choose $x=\frac{a+mb}{2},$ then we get:%
\begin{eqnarray*}
&&\left\vert (1-\lambda )f\left( \frac{a+mb}{2}\right) +\lambda \left( \frac{%
f(a)+f(mb)}{2}\right) -\frac{\Gamma \left( \kappa +1\right) 2^{\kappa -1}}{%
\left( mb-a\right) ^{\kappa }}\left[ J_{\left( \frac{a+mb}{2}\right)
^{-}}^{\kappa }f(a)+J_{\left( \frac{a+mb}{2}\right) ^{+}}^{\kappa }f(mb)%
\right] \right\vert \\
&\leq &\varphi _{4}^{\frac{1}{p}}\left( \kappa ,\lambda ,p\right) \frac{%
\left( mb-a\right) ^{2}}{8\left( \kappa +1\right) }\left\{ \left( \frac{%
\left\vert f^{\prime \prime }\left( \frac{a+mb}{2}\right) \right\vert
^{q}+\alpha m\left\vert f^{\prime \prime }\left( \frac{a}{m}\right)
\right\vert ^{q}}{\alpha +1}\right) ^{\frac{1}{q}}\right. \\
&&+\left. \left( \frac{\left\vert f^{\prime \prime }\left( \frac{a+mb}{2}%
\right) \right\vert ^{q}+\alpha m\left\vert f^{\prime \prime }\left(
b\right) \right\vert ^{q}}{\alpha +1}\right) ^{\frac{1}{q}}\right\}
\end{eqnarray*}

(b) If we choose $x=\frac{a+mb}{2},\ \lambda =\frac{1}{3},$then we get:%
\begin{eqnarray*}
&&\left\vert \frac{1}{6}\left[ f(a)+4f\left( \frac{a+mb}{2}\right) +f(mb)%
\right] -\frac{\Gamma \left( \kappa +1\right) 2^{\kappa -1}}{\left(
mb-a\right) ^{\kappa }}\left[ J_{\left( \frac{a+mb}{2}\right) ^{-}}^{\kappa
}f(a)+J_{\left( \frac{a+mb}{2}\right) ^{+}}^{\kappa }f(mb)\right] \right\vert
\\
&\leq &\varphi _{4}^{\frac{1}{p}}\left( \kappa ,\frac{1}{3},p\right) \frac{%
\left( mb-a\right) ^{2}}{8\left( \kappa +1\right) }\left\{ \left( \frac{%
\left\vert f^{\prime \prime }\left( \frac{a+mb}{2}\right) \right\vert
^{q}+\alpha m\left\vert f^{\prime \prime }\left( \frac{a}{m}\right)
\right\vert ^{q}}{\alpha +1}\right) ^{\frac{1}{q}}\right. \\
&&+\left. \left( \frac{\left\vert f^{\prime \prime }\left( \frac{a+mb}{2}%
\right) \right\vert ^{q}+\alpha m\left\vert f^{\prime \prime }\left(
b\right) \right\vert ^{q}}{\alpha +1}\right) ^{\frac{1}{q}}\right\}
\end{eqnarray*}

(c) If we choose $x=\frac{a+mb}{2},\ \lambda =\frac{1}{3},$ and $\kappa =1,$%
then we get:%
\begin{eqnarray*}
&&\left\vert \frac{1}{6}\left[ f(a)+4f\left( \frac{a+mb}{2}\right) +f(mb)%
\right] -\frac{1}{\left( mb-a\right) }\dint\limits_{a}^{mb}f(x)dx\right\vert
\\
&\leq &\frac{\left( mb-a\right) ^{2}}{16}\varphi _{4}^{\frac{1}{p}}\left( 1,%
\frac{1}{3},p\right) \left\{ \left( \frac{\left\vert f^{\prime \prime
}\left( \frac{a+mb}{2}\right) \right\vert ^{q}+\alpha m\left\vert f^{\prime
\prime }\left( \frac{a}{m}\right) \right\vert ^{q}}{\alpha +1}\right) ^{%
\frac{1}{q}}\right. \\
&&+\left. \left( \frac{\left\vert f^{\prime \prime }\left( \frac{a+mb}{2}%
\right) \right\vert ^{q}+\alpha m\left\vert f^{\prime \prime }\left(
b\right) \right\vert ^{q}}{\alpha +1}\right) ^{\frac{1}{q}}\right\}
\end{eqnarray*}%
where 
\begin{equation*}
\varphi _{4}\left( 1,\frac{1}{3},p\right) =\left( \frac{2}{3}\right)
^{1+2p}\beta \left( 1+p,1+p\right) +\left( \frac{1}{3}\right)
^{1+p}._{2}F_{1}\left( -p,1;p+2;\frac{1}{3}\right) .
\end{equation*}

(d) If we choose $x=\frac{a+mb}{2}$ and$\ \lambda =0,$then we get:%
\begin{eqnarray*}
&&\left\vert f\left( \frac{a+mb}{2}\right) -\frac{\Gamma \left( \kappa
+1\right) 2^{\kappa -1}}{\left( mb-a\right) ^{\kappa }}\left[ J_{\left( 
\frac{a+mb}{2}\right) ^{-}}^{\kappa }f(a)+J_{\left( \frac{a+mb}{2}\right)
^{+}}^{\kappa }f(mb)\right] \right\vert \\
&\leq &\frac{\left( mb-a\right) ^{2}}{16}\left( \frac{1}{p\left( \kappa
+1\right) +1}\right) ^{\frac{1}{p}}\left\{ \left( \frac{\left\vert f^{\prime
\prime }\left( \frac{a+mb}{2}\right) \right\vert ^{q}+\alpha m\left\vert
f^{\prime \prime }\left( \frac{a}{m}\right) \right\vert ^{q}}{\alpha +1}%
\right) ^{\frac{1}{q}}\right. \\
&&+\left. \left( \frac{\left\vert f^{\prime \prime }\left( \frac{a+mb}{2}%
\right) \right\vert ^{q}+\alpha m\left\vert f^{\prime \prime }\left(
b\right) \right\vert ^{q}}{\alpha +1}\right) ^{\frac{1}{q}}\right\}
\end{eqnarray*}

(e) If we choose $x=\frac{a+mb}{2}$ and $\lambda =1,$then we get:%
\begin{eqnarray*}
&&\left\vert \frac{f(a)+f(mb)}{2}-\frac{\Gamma \left( \kappa +1\right)
2^{\kappa -1}}{\left( mb-a\right) ^{\kappa }}\left[ J_{\left( \frac{a+mb}{2}%
\right) ^{-}}^{\kappa }f(a)+J_{\left( \frac{a+mb}{2}\right) ^{+}}^{\kappa
}f(mb)\right] \right\vert \\
&\leq &\frac{\left( mb-a\right) ^{2}}{16}\varphi _{4}^{\frac{1}{p}}\left(
\kappa ,1,p\right) \left\{ \left( \frac{\left\vert f^{\prime \prime }\left( 
\frac{a+mb}{2}\right) \right\vert ^{q}+\alpha m\left\vert f^{\prime \prime
}\left( \frac{a}{m}\right) \right\vert ^{q}}{\alpha +1}\right) ^{\frac{1}{q}%
}\right. \\
&&+\left. \left( \frac{\left\vert f^{\prime \prime }\left( \frac{a+mb}{2}%
\right) \right\vert ^{q}+\alpha m\left\vert f^{\prime \prime }\left(
b\right) \right\vert ^{q}}{\alpha +1}\right) ^{\frac{1}{q}}\right\}
\end{eqnarray*}%
where%
\begin{equation*}
\varphi _{4}\left( \kappa ,1,p\right) =\frac{\left( 1+\kappa \right) ^{\frac{%
p\left( \kappa +1\right) +1}{\kappa }}}{\kappa }\beta \left( \frac{1}{%
1+\kappa };\frac{1+p}{\kappa },1+p\right) .
\end{equation*}

(g) If we choose $x=\frac{a+mb}{2},\lambda =1$ and $\kappa =1,$then we get:$%
\ $%
\begin{eqnarray*}
&&\left\vert \frac{f(a)+f(mb)}{2}-\frac{1}{\left( mb-a\right) }%
\dint\limits_{a}^{mb}f(x)dx\right\vert \\
&\leq &\frac{\left( mb-a\right) ^{2}}{4}\left( 2\beta \left( \frac{1}{2}%
;1+p,1+p\right) \right) ^{\frac{1}{p}}\left\{ \left( \frac{\left\vert
f^{\prime \prime }\left( \frac{a+mb}{2}\right) \right\vert ^{q}+\alpha
m\left\vert f^{\prime \prime }\left( \frac{a}{m}\right) \right\vert ^{q}}{%
\alpha +1}\right) ^{\frac{1}{q}}\right. \\
&&+\left. \left( \frac{\left\vert f^{\prime \prime }\left( \frac{a+mb}{2}%
\right) \right\vert ^{q}+\alpha m\left\vert f^{\prime \prime }\left(
b\right) \right\vert ^{q}}{\alpha +1}\right) ^{\frac{1}{q}}\right\}
\end{eqnarray*}%
where 
\begin{equation*}
2\beta \left( \frac{1}{2};1+p,1+p\right) =\beta \left( 1+p,1+p\right) .
\end{equation*}
\end{corollary}


\begin{thebibliography}{99}
\bibitem{AS65} M. Abramowitz, I.A. Stegun, (eds.), Handbook of Mathematical
Functions with Formulas, Graphs, and Mathematical Tables, Dover, New York
(1965).

\bibitem{DBI13} M.I.Bhatti, M.Iqbal and S.S.Dragomir, Some new fractional
integral Hermite-Hadamard type inequalities, RGMIA Res. Rep. Coll., 16
(2013), Article 2.

\bibitem{I13} I. Iscan, Hermite-Hadamard type inequalities for functions
whose derivatives are $(\alpha ,m)$-convex, International Journal of
Engineering and Applied sciences\textit{,} 2 (3) (2013), 82- 91.

\bibitem{I13b} I. Iscan, A new generalization of some integral inequalities
for $(\alpha ,m)$-convex functions,\textit{\ }Mathematical Sciences, 7 (22)
(2013), 1-8. doi:10.1186/2251-7456-7-22

\bibitem{I13c} I. Iscan, Generalization of different type integral
inequalities for $s$-convex functions via fractional integrals, Applicable
Analysis (2013), 1-17. doi:10.1080/00036811.2013.851785. Available online
at: http://dx.doi.org/10.1080/00036811.2013.851785.

\bibitem{M93} Mihe\c{s}an, V.G. A generalization of the convexity, Seminar
on Functional Equations, Approx. and Convex, Cluj-Napoca, Romania, 1993.

\bibitem{OAK11} M. E. \"{O}zdemir, M. Avci and H. Kavurmaci,
Hermite-Hadamard-type inequalities via $(\alpha ,m)$-convexity,\textit{\ }%
Computers and Mathematics with Applications\textit{,} 61 (2011), 2614--2620.

\bibitem{P12b} J. Park,On the Simpson-like type inequalities for twice
differentiable $(\alpha ,m)$-convex mapping, International Journal of Pure
and Applied Mathematics, 78 (5) (2012), 617--634.

\bibitem{SA11} M.Z. Sarikaya and N. Aktan, On the generalization of some
integral inequalities and their applications, Mathematical and Computer
Modelling\textit{,} 54 (2011) 2175--2182.

\bibitem{SO12} M.Z. Sarikaya and H. Ogunmez, On new inequalities via
Riemann-Liouville fractional integration, Abstract and Applied Analysis,
Volume 2012, Article ID 428983, 10 pages. doi:10.1155/2012/428983

\bibitem{SSYB11} M.Z. Sarikaya, E. Set, H. Yaldiz and N. Basak, Hermite
-Hadamard's inequalities for fractional integrals and related fractional
inequalities, Mathematical and Computer Modelling.
doi:10.1016/j.mcm.2011.12.048.

\bibitem{S12} E. Set, New inequalities of Ostrowski type for mappings whose
derivatives are $s$-convex in the second sense via fractional integrals,
Computers and Mathematics with Applications\textit{,} 63 (2012) 1147--1154.

\bibitem{SY13} M.Z. Sarikaya and H. Yaldiz, On weighted Montogomery
identities for Riemann-Liouville fractional integrals, Konuralp Journal of
Mathematics, 1 (1) (2013), 48--53.

\bibitem{WLFZ12} J.R. Wang, X. Li, M. Fe\v{c}kan and Y. Zhou,
Hermite--Hadamard-type inequalities for Riemann--Liouville fractional
integrals via two kinds of convexity, Applicable Analysis, (2012), 1--13.
doi:10.1080/00036811.2012.727986.
\end{thebibliography}
\end{document}